\documentclass[a4paper,11pt]{article}
\usepackage{geometry} % see geometry.pdf on how to lay out the page. There's lots.
\usepackage{verbatim}
\usepackage{fancyhdr}       %页眉页脚
\usepackage[colorlinks, citecolor=blue]{hyperref}           % hyperreference with colorlinks
\usepackage {graphics}
\usepackage{subfigure}
\usepackage{float}
\usepackage{graphicx}
\usepackage{mathrsfs}
\usepackage{indentfirst, latexsym, bm, color}
\usepackage{amsmath, amssymb, amsfonts,amsthm}
\usepackage{caption}
\usepackage[all]{xy}
\usepackage{setspace}
\usepackage[named]{algorithm}
\usepackage{authblk}

\makeatletter
\newcommand*\bigcdot{\mathpalette\bigcdot@{.7}}
\newcommand*\bigcdot@[2]{\mathbin{\vcenter{\hbox{\scalebox{#2}{$\m@th#1\bullet$}}}}}
\makeatother

\theoremstyle{plain}
\theoremstyle{definition}
\newtheorem{definition}{Definition}
\newtheorem{lemma}[definition]{Lemma}
\newtheorem{theorem}[definition]{Theorem}
\newtheorem{proposition}[definition]{Proposition}
\newtheorem{corollary}[definition]{Corollary}
\newtheorem{example}[definition]{Example}
\newtheorem{remark}[definition]{Remark}
\newtheorem{conjecture}[definition]{Conjecture}

\title{Periodic multivariate formal power series}
\author{Xue Zhang \thanks{xz.math@outlook.com}}
\begin{document}

\maketitle

\begin{abstract}
A system of multivariate formal power series $\varphi$ with a homogeneous decomposition $\varphi=\sum_{k=0}^\infty\varphi_k$ is invertible under composition if $\varphi_0=0$ and $\mathrm{det}(\varphi_1)\ne 0.$ All invertible series over a field $K$ form a formal transformation group $G_\infty(n,K).$ We prove that every periodic series $\varphi\in G_\infty(n,K)$ with $\varphi_1$ diagonalizable is conjugate to $\varphi_1.$ This classifies all periodic series in $G_\infty(n,\mathbb{C}).$ A constraint for a periodic series is obtained when its first term is a multiple of identity.
\end{abstract}


%\tableofcontents
\section{Introduction}
The study of formal power series involves many areas of mathematics, for example analysis \cite{HCartan1963,Raney1960}, combinatorics \cite{Gessel1987,Labelle1981}, dynamical systems \cite{Lubin1994}, differential equations \cite{Zhaowh2005,Parvica2005} and the famous Jacobian conjecture \cite{HBass1982,Bondt2004,VEssen2021}. There are two kinds of multiplication of formal power series, usual multiplication and  composition. This note deals with periodic series in the group of formal power series in multiple variables with the composition law. 

 Given a commutative ring $R$ and a system of formal power series $\varphi=(f_{1},\cdots,f_{n})$ in $R[[x_{1},\cdots,x_{n}]]^{n}$ of the form $\varphi=\sum_{k=0}^{\infty}\varphi_{k}$ with $\varphi_{k}$ consisting of terms of degree $k.$ If $\varphi_{0}=0$ and $\varphi_1=(x_1,\cdots,x_n),$ then there is a unique series $\phi=(g_{1},\cdots,g_{n})\in R[[x_{1},\cdots,x_{n}]]^{n}$ such that $\varphi\circ\phi=\phi\circ\varphi=id,$ i.e. 
$$(f_{1}(\phi),\cdots,f_n(\phi))=(g_{1}(\varphi),\cdots,g_{n}(\varphi))=(x_{1},\cdots,x_{n}).$$
All invertible series in $R[[x_{1},\cdots,x_{n}]]^{n}$ form a non-abelian group, called \emph{formal transformation group}. The unit of this group is the identity $id=(x_{1},\cdots,x_{n}).$ The inverse of a invertible series is given by the Lagrange inversion formula, a family of integer polynomials that is independent of $R.$ There are many proofs of it, including Raney \cite{Raney1960}, Labelle \cite{Labelle1981},  Gessel \cite{Gessel1987}, Haiman-Schmitt \cite{Haiman1989} and  Johnston-Prochno \cite{Johnston2022}. In addition,  see \cite{Gessel1980} for the non-commutative case.

Let $R=K$ be a field and let $GL_n(K)$ denote the general linear group over $K.$ We consider the formal transformation group $G_\infty(H)$ of formal power series over $K$ with respect to a subgroup $H$ of $GL_n(K),$ which is given by 
$$G_{\infty}(H):=\left\{\varphi=\sum_{i=0}^\infty\varphi_i\in K[[x_{1},\cdots,x_{n}]]^{n}|\ \varphi_0=0,\varphi_1\in H\right\}.$$
Here $\sum_{i=0}^\infty\varphi_i$ is the unique homogeneous decomposition of $\varphi.$ When $H=GL_n(K),$ we write $G_\infty(n,K)$ for $G_\infty(GL_n(K)).$

A series in $G_\infty(n,K)$ is said to be \emph{periodic} if its order in $G_\infty(n,K)$ is finite and is said to be \emph{$m$-periodic} if its order equals to $m.$ For $n=1,$ it is well known that $\varphi=\omega_m x+\sum_{k=2}^\infty a_kx^k$ is $m$-periodic if and only if $\varphi$ is conjugate to $\omega_m x,$ where $\omega_m$ is a primitive $m$-th root of unit in $K.$ See, for example, \cite{SalomonB1948},\cite{Stephen1970}, \cite{Cheon2013} and \cite{Anthony2015}. Brewer \cite{Thomas2014} generalized this result to high dimension and proved that every $m$-periodic series $\varphi$ with $\varphi_1=\omega_m\cdot id$ is conjugate to $\varphi_1,$ where $\omega\cdot id$ is a multiple of the identity. 

This note generalizes Brewer's result to a more general case where $\varphi_1$ is a diagonalizable matrix.  Let $\mathrm{char}\,K$ denote the characteristic of $K.$ In general, two elements $g_1,g_2$ in a group $G$ are said to be \emph{conjugate} to each other if $g_1h=hg_2$ for some $h\in G.$ Now we state our main results.
\begin{theorem}\label{intro01}
If $\mathrm{char}\,K=0$ and $\varphi\in G_\infty(n,K)$ is periodic, then $\mathrm{ord}(\varphi)=\mathrm{ord}(\varphi_1).$	
\end{theorem}
\begin{theorem}
Suppose $m\geq 2$ and $m\notin \mathrm{char}\,K\cdot\mathbb{Z}.$ Let $\varphi\in G_\infty(n,K)$ be a $m$-periodic series with $\varphi_1$ diagonalizable. Then $\varphi$ is conjugate to $\varphi_1,$ i.e. there exists $\phi\in G_\infty(n,K)$ such that $\varphi=\phi^{-1}\varphi_1 \phi.$
\end{theorem}
Next result says that if the first term of $\varphi$ is a multiple of the identity, then its $m$-periodicity is completely determined by a constraint.
\begin{theorem}\label{intthm032}
Suppose $m\geq 2$ and $m\notin \mathrm{char}\,K\cdot\mathbb{Z}.$ Let $\omega$ be a primitive $m$-th root of unit in $K.$ Let $\varphi=\omega\cdot id+\sum_{i=2}^\infty\varphi_i\in G_\infty(n,K).$ Then $\varphi$ is $m$-periodic if and only if $\varphi_{mk+1}=-\frac{\omega}{m}((\varphi_{\leq mk})^m)_{mk+1}$ for any $k\geq 1,$ where $\varphi_{\leq mk}=\sum_{i=1}^{mk}\varphi_i.$
\end{theorem}
\begin{remark}
By reducing to $SL_n(\mathbb{Z}),$ it is not difficult to see that $G_\infty(n,\mathbb{Z})$ has no $m$-periodic series for $m> 2.$ As an interesting consequence of Theorem \ref{intthm032}, there are $2$-periodic series in $G_\infty(n,\mathbb{Z})$ since the coefficients in $((\varphi_{\leq 2k})^2)_{2k+1}$ are all even when the constraint holds.
\end{remark}
 We get an immediate consequence of Theorem \ref{intthm032}.
\begin{corollary}
	If $\varphi\in G_\infty(n,\mathbb{C})$ is periodic, then $\varphi$ is conjugate to $\varphi_1.$
\end{corollary}
This note is arranged as follows. In section \ref{firsectintr67} we review background material from \cite{Dorfer1999,ArnoEssen2000,Farrell2008,Thomas2014}, including formal transformation group and its basic properties. In section \ref{prfmainthm53}, we give a recursive description (Theorem \ref{inverseformula66}) of the Lagrange inversion formula and some lemmas. In section \ref{secsubgroups3}, we introduce a kind of subgroup $G_{p,q}(n,K)$ of $G_\infty(n,K)$ depending on a pair of integers $p$ and $q.$ In section \ref{periptssect6}, we classify periodic series with an initial condition and prove our main results.

\section{Notation and conventions}\label{firsectintr67}
We begin by recalling the basic properties of formal transformation groups. Let $K$ be a field and let $K\{x_1,\cdots,x_n\}$ (or simply denoted by $K[[X]]$) be the formal power series ring in $n$ variables $x_1,\cdots,x_n$ over $K$ whose element has the form
$$\sum_{\alpha} a_\alpha x_1^{\alpha_1}\cdots x_n^{\alpha_n},$$
where $a_\alpha\in K$ and the index $\alpha=(\alpha_1,\cdots,\alpha_n)$ runs over $\mathbb{N}^n.$ For each $k\in\mathbb{N},$ we consider a set $K[X]_k^n$ that consists of all  polynomial vectors $(f_1,\cdots,f_n)\in K[X]^n$ where each $f_i$ is a homogeneous polynomial of degree $k.$ The first space $K[X]_0^n$ is just $K^n.$ The second space $K[X]_1^n$ is isomorphic to the matrix space $M_{n\times n}(K)$ over $K$ by the jacobian map
\begin{eqnarray}\label{matrxkxn198}
\left(\sum_{j=1}^na_{1j}x_j,\cdots,\sum_{j=1}^na_{nj}x_j\right)\mapsto (a_{ij}).
\end{eqnarray}
With above notations, every series $\varphi\in K[[X]]^n$ has a unique \emph{homogeneous decomposition} 
$$\varphi=\sum_{k=0}^\infty\varphi_k$$
with $\varphi_k\in K[X]_k^n.$ For convenience, we write $\varphi_{\leq m}$ for $\sum_{i=0}^m\varphi_i.$ Also, we use the notation $K[[X]]_{\leq a}^n$ for $\bigoplus_{i=0}^a K[X]_i^n.$ We always use Greek letters to represent series in $K[[X]]^n$ and use $f$ or $g$ to represent the components of series. Unless otherwise specified, these notations will be used throughout this paper. 

Note that for two series $\varphi=(f_1,\cdots,f_n),\phi\in K[[X]]^n,$ the series composition 
$$\varphi\circ\phi=(f_1(\phi),\cdots,f_n(\phi))$$
 is a product of $K[[X]]^n.$ It is easy to see that the set $G_\infty(n,K)$ defined in the introduction is multiplicatively closed. The classical result (see, for example \cite[Theorem 1.1.2]{ArnoEssen2000}) about formal series shows that every element in $G_\infty(n,K)$ has a unique inverse, that is, $G_\infty(n,K)$ is an infinite dimensional group. From now on, we always write $\varphi\phi$ for the composition $\varphi\circ\phi.$ Since $\varphi_1$ is just the jacobian matrix of $\varphi$ by (\ref{matrxkxn198}), the map $(\cdot)_1:G_\infty(n,K)\rightarrow K[X]_1^n$ is indeed a group homomorphism, i.e. the equality $(\varphi\phi)_1=\varphi_1\phi_1$ holds for all $\varphi,\phi\in G_\infty(n,K).$
 
Next we explain that $G_\infty(n,K)$ is a limit of a sequence of finite-dimensional Lie groups. For each positive integer $k,$ let 
 $$S_k:=\{\varphi\in K[[X]]^n|\ \varphi_i=0\ \mbox{for every }0\leq i\leq k\}.$$
  For convenience, we also write $S_\infty=0$ to match $G_\infty(n,K)=G_\infty(n,K)\ \mathrm{mod}\ S_\infty.$ We say $\varphi\equiv \phi\ \mathrm{mod}\ S_k$ if $\varphi_j=\phi_j$ for all $j\leq k.$  
 \begin{definition}
 For each $k\geq 1,$ we define the $k$-th transformation group over $K$ of order $n$ to be
 $$G_k(n,K):=G_\infty(n,K)\ \mathrm{mod}\ S_k.$$
 \end{definition}
 The product of $\varphi=(f_1,\cdots,f_n),\phi=(g_1,\cdots,g_n)\in G_k(n,K)$ is given by
$$\varphi\phi=(f_1(g_1,\cdots,g_n),\cdots,f_n(g_1,\cdots,g_n))\ \mathrm{mod}\ S_{k}$$
that is well-defined. The identity element $id$ is $(x_1,\cdots,x_n).$ 

When $k<\infty,$ a direct calculation shows that the dimension of $G_k(n,K)$ over $K$ is 
$$\mathrm{dim}\,G_k(n,K)=n\sum_{i=1}^k\binom{n+i-1}{i}.$$
If $k=1,$ the initial transformation group $G_1(n,K)$ is just the general linear group $GL_n(K).$ So it is a generalization of the general linear group. There is a natural homomorphism between any two transformation groups
 $$\pi_{k,m}:G_k(n,K)\rightarrow G_m(n,K),\ \varphi\mapsto\varphi\ \mathrm{mod}\ S_{m}$$
for $k\geq m.$ In particular, $\pi_{k,1}(\varphi)=\varphi_1.$ If $\varphi$ is a polynomial vector with $\varphi(0)=0,$ then $\varphi_1$ can be viewed as the tangent map of $\varphi:K^n\rightarrow K^n$ at $0.$ 
\begin{remark}
There are other groups of special formal power series, for example, the group of formal symmetric series \cite{Bondt2004,VEssen2021} and the group of formal Laurent series \cite{Ganxx2011}.
\end{remark}
The infinite dimensional group $G_\infty(n,K)$ is actually the inverse limit of $G_k(n,K),$ that is to say,  $G_\infty(n,K)$ is isomorphic to 
$$\lim_{\longleftarrow}G_k(n,K)=\{(A_1,A_2,\cdots)|A_j\in G_j(n,K),\pi_{i,j}(A_i)=A_j,\forall i>j\}.$$
\begin{remark}
We attach a positive integer $T(\varphi)$ to each $\varphi\in G_\infty(n,K)$ with $\varphi\ne \varphi_1$ by 
$$T(\varphi)=\sup\left\{m\in\mathbb{N}|\ \varphi=\sum_{k=0}^\infty \varphi_{mk+1}\right\}.$$
If $\varphi= \varphi_1,$ define $T(\varphi)=0.$ For instance, {in $n=1$ dimension,} $T(x+x^2)=1$ and $T(x+x^3+x^5)=2.$ By the recurrence relations (\ref{recurelatin666}), one can see that $T(\varphi)=T(\varphi^{-1}).$
\end{remark}
\begin{example}
The inverse of a quadratic function $\varphi=x+ax^2\in G_\infty(1,K)$ with $a\ne 0$ is given by $$\varphi^{-1}=\frac{\sqrt{4 a x+1}-1}{2 a}=x-a x^2+2 a^2 x^3-5 a^3 x^4+14 a^4 x^5-42 a^5 x^6+\cdots$$ 
\end{example}
At the end of this section we describe the product of $G_\infty(n,K)$ in another way. Now assume $\mathrm{char}\,K=0.$ For any $k\geq 2,$ consider the linear space $\bigoplus_{i=2}^k K[X]_i^n$
over $K$ with a set of monomials $e_1,\cdots,e_s$ as its basis. Then the product of $G_k(n,K)$ gives an integer polynomial vector $P_k=(P_{k,1},\cdots,P_{k,s})\in\mathbb{Z}[x_1,\cdots,x_s,y_1,\cdots,y_s]^s$ by
$$id+\sum_{i=1}^sP_{k,i}(x_1,\cdots,x_s,y_1,\cdots,y_s)e_i\equiv\left(id+\sum_{i=1}^s x_ie_i\right)\left(id+\sum_{i=1}^s y_ie_i\right)\ \mathrm{mod}\ S_k$$
with respect to the basis $e_1,\cdots,e_s,$ which is independent of the field $K.$ The existence of {inverses in $G_k(n,K)$} implies that there exists an integer polynomial vector $I_k=(I_{k,1},\cdots,I_{k,s})\in\mathbb{Z}[x_1,\cdots,x_s]^s$ defined by the inverse equation
$$id+\sum_{i=1}^sI_{k,i}(x_1,\cdots,x_s)e_i\equiv\left(id+\sum_{i=1}^s x_ie_i\right)^{-1}\ \mathrm{mod}\ S_k$$
satisfying 
$$P_k(x_1,\cdots,x_s,I_k(x_1,\cdots,x_s))=P_k(I_k(x_1,\cdots,x_s),x_1,\cdots,x_s)=0.$$ 
{ In fact, the collection of all polynomial vectors $\{P_k\}_{k=2}^\infty$ determines $\{I_k\}_{k=2}^\infty,$ see Theorem \ref{inverseformula66} in section \ref{prfmainthm53}. Hence $I_k$ is also independent of $K.$} These polynomial vectors $P_k$ and $I_k$ are related to the incidence algebra \cite{Doubilet1972,Haiman1989}.

\section{{A recursive formula for inversion in $G_\infty(n,K)$}}\label{prfmainthm53}
In this section we give a recursive description of the inverse of an invertible series, which is simpler than Lagrange inversion formula \cite{Gessel1987,Johnston2022}.
\begin{theorem}\label{inverseformula66}
For every $\varphi\in G_\infty(n,K),$ the inverse $\phi=\sum_{i=1}^\infty\phi_i$ of $\varphi$ is given by the following recurrence relations, 
\begin{eqnarray}\label{recurelatin666}
	\phi_1=\varphi_1^{-1},\ \ \phi_{k+1}=-\varphi_1^{-1}\left(\varphi\left(\sum_{i=1}^k\phi_i\right)\right)_{k+1}.
\end{eqnarray}
\end{theorem}

{We next give some lemmas which will be used in the proofs of Theorem \ref{inverseformula66} and other results of the present paper. A similar method has been used in \cite{Thomas2014}. 
\begin{lemma}\label{lema2ptysum32}
Let $\varphi,\phi,\psi\in K[[X]]^n.$ We have
\begin{itemize}
\item[(\romannumeral1)] $\varphi_1(\phi+\psi)=\varphi_1\phi+\varphi_1\psi.$ 
\item[(\romannumeral2)] $(\phi+\psi)\varphi=\phi\varphi+\psi\varphi.$
\end{itemize}
\end{lemma}
\begin{proof}
This lemma immediately follows by a direct calculation.
\end{proof}

 Lemma \ref{lema2ptysum32} tells us that the right-multiplication of $K[[X]]^n$ and the left-multiplication by $\varphi_1\in K[X]_1^n$ are linear. We then turn our attention to the left-multiplication of $K[[X]]^n.$ 

Fix a nonzero degree $k$ homogenous series $\varphi.$ For each $j\in\mathbb{N}_{\geq 1},$ the map
$$K[[X]]_{\geq 1}^n\rightarrow K[X]_j^n,\ \ \phi\mapsto(\varphi\phi)_j$$
 can be viewed as a map of infinite variables
$$\Phi_j:\prod_{i=1}^\infty K[X]_i^n\rightarrow K[X]_j^n,\ \ \Phi_j(\phi_1,\phi_2,\cdots)=(\varphi(\phi_1+\phi_2+\cdots))_j.$$
Obviously, each $\Phi_j$ has only finite nontrivial variables, that is, there exists an integer $r\in\mathbb{N}$ so that $$\Phi_j(\phi_1,\phi_2,\cdots)=\Phi_j(\phi_1,\cdots,\phi_r,0,0,\cdots)$$
 for any $(\phi_1,\phi_2,\cdots)\in\prod_{i=1}^\infty K[X]_i^n.$ {In other words, $(\varphi\phi)_j=(\varphi\phi_{\leq r})_j$ holds for any $\phi\in K[[X]]_{\geq 1}^n.$} Such the smallest integer $r$ depends only on $j$ and $k.$ We shall denote the smallest integer $r$ by $r(k,j).$ It is not hard to see that 
\begin{eqnarray}\label{smestskj31}
 r(k,j)=\mathrm{max}(j-k+1,0).
\end{eqnarray}
\begin{lemma}\label{twocompdeg6}
Let $\varphi,\phi\in G_\infty(n,K)$ be two series with homogeneous decompositions $\varphi=\sum_{i=1}^\infty\varphi_i$ and $\phi=\sum_{i=1}^\infty\phi_i,$ respectively. Then 
\begin{eqnarray*}
(\varphi\phi)_{\leq m}= \sum_{i=1}^m(\varphi_i\phi_{\leq m+1-i})_{\leq m}.
\end{eqnarray*}
\end{lemma}
\begin{proof}
This lemma follows immediately from the property (\romannumeral2) in Lemma \ref{lema2ptysum32} and (\ref{smestskj31}).
\end{proof}
\begin{lemma}\label{zhankai01}
	For any $\varphi,\phi\in K[[X]]^n,$ we have
	$$(\varphi\phi)_{\leq k+1} = (\varphi_{\leq k}\phi_{\leq k})_{\leq k+1}+\varphi_1\phi_{k+1}+\varphi_{k+1}\phi_1.$$
\end{lemma}
\begin{proof}
By Lemma \ref{lema2ptysum32} and Lemma \ref{twocompdeg6}, the assersion follows from
\begin{equation*}
\begin{split}
(\varphi\phi)_{\leq k+1}=&\ (\varphi_{\leq k}\phi)_{\leq k+1}+(\varphi_{k+1}\phi)_{\leq k+1}\\
=&\  (\varphi_{\leq k}\phi_{\leq k})_{\leq k+1}+\varphi_1\phi_{k+1}+\varphi_{k+1}\phi_1,
\end{split}
\end{equation*}	
where the last equality holds due to the fact $(\varphi_1\phi_{\leq k})_{k+1}=0.$
\end{proof}

Now we are ready for the proof of Theorem \ref{inverseformula66}.

\emph{Proof of Theorem \ref{inverseformula66}.} Now given a series $\varphi\in G_\infty(n,K).$ Let us first calculate the necessary condition for $\varphi\phi=id$ in $G_\infty(n,K).$ By $(\varphi\phi)_1=\varphi_1\phi_1=id,$ we see that $\phi_1=\varphi_1^{-1},$ since $G_1(n,K)\cong GL_n(K)$ is a group. Suppose we are given $\phi_1,\phi_2,\cdots,\phi_k$ with $\phi_i\in K[X]_i^n$ such that $\phi=\sum_{i=1}^k\phi_i$ satisfies 
$$(\varphi\phi)_{\leq k}=(\phi\varphi)_{\leq k}= id.$$
 Now we solve the equation 
 $$(\varphi(\phi+\phi_{k+1}))_{\leq k+1}= ((\phi+\phi_{k+1})\varphi)_{\leq k+1}= id$$
  about $\phi_{k+1}\in K[X]_{k+1}^n.$ By the property (\romannumeral1) in Lemma \ref{lema2ptysum32}, {left multiplication by $\varphi_1$ is linear, and hence it follows from Lemma \ref{twocompdeg6} that} 
\begin{equation*}
	\begin{split}
	(\varphi(\phi+\phi_{k+1}))_{\leq k+1}=&\ (\varphi_1(\phi+\phi_{k+1}))_{\leq k+1}+\sum_{i=2}^{k+1}(\varphi_i\phi)_{\leq k+1}\\
	= &\ \varphi_1\phi_{k+1}+(\varphi\phi)_{\leq k+1} .
	\end{split}
\end{equation*}
As $(\varphi(\phi+\phi_{k+1}))_{k+1}=0,$ we obtain $\phi_{k+1}=-\varphi_1^{-1}\cdot(\varphi\phi)_{k+1}.$ This finishes the proof.}

Next, we give an interesting corollary of Theorem \ref{inverseformula66}:
\begin{corollary}\label{imtsidlem1}
Let $\varphi,\phi\in G_\infty(n,K)$ be two series with $\varphi_1=\phi_1=id,$ satisfying 
$$(\varphi\phi)_{\leq k}=(\phi\varphi)_{\leq k}= id$$
for some $k\in\mathbb{N}_{\geq 1}.$ Then $(\varphi\phi)_{k+1}=(\phi\varphi)_{k+1}.$
\end{corollary}
\begin{proof}
By Lemma \ref{zhankai01}, we find 
\begin{equation}\label{zx3141590}
	(\varphi\phi)_{k+1} = (\varphi_{\leq k}\phi_{\leq k})_{k+1}+\phi_{k+1}+\varphi_{k+1},
\end{equation}	
 Assume $\psi\in G_\infty(n,K)$ is the inverse of $\varphi.$ It follows from Theorem \ref{inverseformula66} that 
 $\phi_{\leq k}=\psi_{\leq k}.$ Apply (\ref{zx3141590}) to the products $\varphi\psi$ and $\psi\varphi,$ respectively, we see that
\begin{eqnarray*}
&&0=(\varphi \psi)_{k+1}=(\varphi\phi)_{k+1}+\psi_{k+1}-\phi_{k+1},\\
&&0=(\psi\varphi)_{k+1}=(\phi\varphi)_{k+1}+\psi_{k+1}-\phi_{k+1},
\end{eqnarray*}
and hence $(\varphi\phi)_{k+1}=(\phi\varphi)_{k+1}.$
\end{proof}

\section{Subgroups of $G_k(n,K)$}\label{secsubgroups3}
In this section we study a special kind of subgroups of $G_k(n,K).$ Now that the projection $\pi_{k,m}$ is a homomorphism for $k\geq m,$ each $G_k(n,K)$ has $k-1$ natural subgroups. Sometimes we write $G_k$ for $G_k(n,K).$

Fix a dimension $n$ and a field $K.$ For each $k\in\mathbb{N}\cup\{\infty\}$ and $1\leq m\leq k,$ define the $m$-th subgroup $G_{k,m}$ of $G_k$ by 
$$G_{k,m}:=\mathrm{Ker}\,\pi_{k,m}=\{\varphi\in K[[X]]^n/S_k|\ \varphi_{\leq m}\equiv id\ \mathrm{mod}\ S_m\}.$$
If we define $G_{k,0}:=G_k$ alone, then every transformation group $G_k$ has a filtration
$$id=G_{k,k}\subset G_{k,k-1}\subset G_{k,k-2}\subset\cdots\subset G_{k,1}\subset G_{k,0}=G_k.$$
We remark that the dimension of $G_{k,m}(n,K)$ over $K$ is $n\sum_{i=m+1}^k\binom{n+i-1}{i}.$ In particular, $\mathrm{dim}\,G_{k,m}(1,K)=k-m.$ 

\begin{theorem}\label{propabelgkk}
Let $k,m$ be integers with $k>m\geq \frac{k}{2}>0.$ Then the group $G_{k,m}$ is commutative.
\end{theorem}
\begin{proof}
We prove this assertion by induction on $k-m.$ First, when $k-m=1,$ for any $\varphi,\phi\in G_{k,k-1}$ with homogeneous decomposition $\varphi=id+\varphi_k,\phi=id+\phi_k,$ it follows that
\begin{equation*}
\begin{split}
\varphi\phi  =&\  id+(\varphi\phi)_k=id+((id+\varphi_k)\phi)_k\\
=&\ id+ (\phi+\varphi_k\phi_1)_k=id+\phi_k+\varphi_k = \phi\varphi.
\end{split}
\end{equation*}
Hence $G_{k,k-1}$ is commutative. Assume now that the group $G_{k,k-r}$ is commutative for any $k,r$ with $1\leq r\leq [\frac{k}{2}].$ We then turn to the case $k-m=r+1$ with $r+1\leq[\frac{k}{2}].$ For any $\varphi,\phi\in G_{k,k-r-1}$ with $\varphi=id+\sum_{i=0}^r\varphi_{k-i},\phi=id+\sum_{i=0}^r\phi_{k-i},$ it follows by Lemma \ref{twocompdeg6} that
\begin{equation*}
(\varphi\phi)_{\leq k}=\varphi_k+\phi_k+(\varphi_{\leq k-1}\phi_{\leq k-1})_{\leq k}.
\end{equation*}
Notice that $(\varphi_{\leq k-1}\phi_{\leq k-1})_k=0$ since $2r+1<k.$ Then 
$$(\varphi_{\leq k-1}\phi_{\leq k-1})_{\leq k}=(\varphi_{\leq k-1}\phi_{\leq k-1})_{\leq k-1}.$$ 
Using the assumption that $G_{k-1,k-r-1}$ is commutative, we have 
$$(\varphi_{\leq k-1}\phi_{\leq k-1})_{\leq k-1}=(\phi_{\leq k-1}\varphi_{\leq k-1})_{\leq k-1},$$ which implies that $(\varphi\phi)_{\leq k}=(\phi\varphi)_{\leq k}.$ Hence $G_{k,k-r-1}$ is commutative. This finishes the proof by induction on $k-m.$
\end{proof}

{In particular}, Theorem \ref{propabelgkk} implies that $G_{k,k-1},k>1,$ as a group, is isomorphic to the vector space $K[X]_k^n.$ Moreover, for any $\varphi\in G_{k,k-1},$ it follows that
\begin{eqnarray}\label{powerfork56}
(\varphi^m)_{\leq k}= id+m\cdot \varphi_k.
\end{eqnarray}

Next we list all abelian subgroups of $G_k(1,K).$ The linear structure of $K^m$ makes it an abelian group with the operator '+'. However $K^2$ has another group structure. Now we define a product on $K^2$ associated with a constant $\lambda\in K$ by
$$(x_1,y_1)\circ(x_2,y_2)=(x_1+x_2,y_1+y_2+\lambda x_1x_2).$$
It is not difficult to verify that
$$(x_1,y_1)\circ(x_2,y_2)\circ\cdots\circ(x_s,y_s)=\left(\sum x_i,\sum y_i+\lambda\sum_{i<j}x_ix_j\right).$$
Hence $(K^2,\circ)$ is an abelian group, and is denoted by $\Omega_\lambda.$ In fact, $\Omega_\lambda$ is isomorphic to $K^2$ via a group isomorphism 
$$(K^2,+)\rightarrow (K^2,\circ),\ \  (x,y)\mapsto (x,y+\frac{1}{2}\lambda x^2).$$

If $m\geq 1$ and $n=1,$ every polynomial vector in $G_{k,m}$ has the form $x+\sum_{i=m+1}^k a_ix^i.$ So we can identify $G_{k,m}$ with $K^{k-m}$ in a natural way. 
\begin{proposition}
When $n=1,$ all abelian groups in $\{G_{i,j}\}_{i\geq j}$ are given as follows:
\begin{itemize}
	\item[(\romannumeral1)] $G_{1,0}=K^*,$
	\item[(\romannumeral2)] $G_{2k,m}=K^{2k-m}$ with  $2k\geq m\geq k\geq 1,$
	\item[(\romannumeral3)] $G_{2k+1,k}=K^{k-1}\oplus \Omega_{k+1}$ with $k\geq 1,$
	\item[(\romannumeral4)] $G_{2k+1,m}=K^{2k+1-m}$ with $2k\geq m-1\geq k\geq 1.$
\end{itemize}
\end{proposition}
\begin{proof}
The proof is omitted here. It is easy to check them one by one. 	
\end{proof}
For example, let $\varphi=x+\sum_{i=k+1}^{2k}a_ix^i,\phi=x+\sum_{i=k+1}^{2k}b_ix^i\in G_{2k,k}(1,K),$ then 
$$\varphi\phi\equiv \phi\varphi\equiv \varphi+\phi \ \mathrm{mod}\ x^{2k+1}.$$

\section{Periodic series}\label{periptssect6}
Recall that an element $g$ in an infinite group $G$ is periodic if there exists a positive integer $s$ such that $g^s=e,$ where $e$ denotes the identity element of $G,$ and let $\mathrm{ord}(g)$ denote the order of $g.$ We say $g\in G$ is \emph{$p$-periodic} if $\mathrm{ord}(g)=p.$

As mentioned in the introduction, the result about periodic series in $G_\infty(1,K)$ is well known, see \cite{SalomonB1948,Stephen1970,Cheon2013,Anthony2015}. We will discuss the general case $n\geq 2.$ First we observe the following conclusion:

\begin{theorem}\label{char0oreuqs}
If $\mathrm{char}\,K=0$ and $\varphi\in G_\infty(n,K)$ is periodic, then $\mathrm{ord}(\varphi)=\mathrm{ord}(\varphi_1).$
\end{theorem}	
\begin{proof}
We first prove the special case $\mathrm{ord}(\varphi_1)=1.$ Now suppose $\varphi^s=id$ for some positive integer $s,$ and let 
$\varphi=\sum_{i=1}^\infty\varphi_i$ be the homogeneous decomposition of $\varphi$ with $\varphi_1=id.$ We intend to prove $\varphi_k=0$ for any $k\geq 2$ by induction on $k.$ 
\begin{itemize}
	\item[(\romannumeral1)] When $k=2,$ it follows from (\ref{powerfork56}) that $0=(\varphi^s)_2=s\cdot \varphi_2,$ and hence $\varphi_2=0.$ 
	\item[(\romannumeral2)] If $\varphi_2=\varphi_3=\cdots=\varphi_k=0$ for $k\geq 2,$ then $id+\varphi_{k+1}\in G_{k+1,k}.$ Using (\ref{powerfork56}) again, we see that $0=(\varphi^s)_{k+1}=s\cdot\varphi_{k+1},$ which leads to $\varphi_{k+1}=0.$
\end{itemize}

We then prove the general case. Put $s=\mathrm{ord}(\varphi)$ and $t=\mathrm{ord}(\varphi_1).$ Since the projection $\pi_{\infty,1}:G_\infty\rightarrow G_1$ is a homomorphism, $t$ must divide $s.$ Let $\phi=\varphi^t.$ Then $\phi_1=\varphi_1^t=id$ and $\phi$ is periodic since $\phi^{\frac{s}{t}}=\varphi^s=id.$ Now apply the above conclusion of the special case to $\phi,$ we find that $\phi=id,$ that is, $\varphi^t=id.$
\end{proof}
\begin{remark}
From the above proof, it is not difficult to see that Theorem \ref{char0oreuqs} still holds for a commutative ring without torsion.
\end{remark}

\subsection{Orbits of periodic series}
In this section we always assume that the period $p>1$ and $p\notin \mathrm{char}\,K\cdot\mathbb{Z}.$ Let $\omega\in K$ be a primitive $p$-th root of unit in $K$, i.e. $\mathrm{ord}(\omega)=p$ in the multiplicative group $K^*.$ Every diagonal matrix in $GL_n(K)$ of order $p$ has the form $\mathrm{diag}(\omega^{\lambda_1},\cdots,\omega^{\lambda_n}),$ where $\lambda=(\lambda_1,\cdots,\lambda_n)\in\mathbb{Z}^n$ is coprime to $p.$ We try to find all $p$-periodic series $\varphi$ in $G_\infty(n,K)$ with the initial condition that $\varphi_1$ is a diagonal matrix.

We first give a simple lemma.
\begin{lemma}\label{lemfaneq53}
For any $\varphi\in K[[X]]^n$ and integers $m,k\geq 1,$ we have 
$$(\varphi^{m+1})_{k+1}=(\varphi^m)_{k+1}\varphi_1+[(\varphi^m)_{\leq k}\varphi]_{k+1}.$$
\end{lemma}
\begin{proof}
By the fact $(\varphi^m)_{\leq k+1}=(\varphi^m)_{\leq k}+(\varphi^m)_{k+1}$ and Lemma \ref{lema2ptysum32}, one has 
\begin{eqnarray}\label{eqliy749}
(\varphi^m)_{\leq k+1}\varphi=(\varphi^m)_{\leq k}\varphi+(\varphi^m)_{k+1}\varphi.
\end{eqnarray}
Notice that $[(\varphi^m)_{\leq k+1}\varphi]_{k+1}=(\varphi^{m+1})_{k+1},$ hence the lemma holds by taking terms of degree $k+1$ of the equality (\ref{eqliy749}).
\end{proof}

 Suppose $\varphi=\varphi_1+\sum_{i=2}^k\varphi_i$ satisfies $(\varphi^p)_{\leq k}=id$ for some $k\geq 1.$ Now we solve the equation $((\varphi+\varphi_{k+1})^p)_{k+1}=0$ about $\varphi_{k+1}\in K[X]_{k+1}^n.$ Put $\phi=\varphi+\varphi_{k+1}.$ For the fixed integers $p$ and $k,$ we explore the relationship between $(\phi^{m+1})_{k+1}$ and $(\phi^{m})_{k+1}.$ It follows from Lemma \ref{zhankai01} and Lemma \ref{lemfaneq53} that  
\begin{equation}\label{eqammp1ps}
\begin{split}
(\phi^{m+1})_{k+1}&=(\phi^m\phi)_{k+1}=((\varphi^m)_{\leq k}\varphi)_{k+1}+\varphi_1^m\varphi_{k+1}+(\phi^m)_{k+1}\varphi_1    \\
&=(\varphi^{m+1})_{k+1}-(\varphi^m)_{k+1}\varphi_1+\varphi_1^m\varphi_{k+1}+(\phi^m)_{k+1}\varphi_1
\end{split}
\end{equation}
By induction on $m,$ the equality (\ref{eqammp1ps}) gives
{
\begin{eqnarray}\label{relationpps33}
(\phi^m)_{k+1}=(\varphi^m)_{k+1}+\sum_{i=0}^{m-1}\varphi_1^i\varphi_{k+1}\varphi_1^{m-1-i}
\end{eqnarray}
}
for any $m\geq 1.$ Here $\varphi_1^0=id.$ {Putting $m=p$ in equality (\ref{relationpps33}),} we see that $\varphi\in G_\infty(n,K)$ is $p$-periodic if and only if 
\begin{eqnarray}\label{iffppts56}
((\varphi_{\leq k})^p)_{k+1}+\sum_{i=0}^{p-1}\varphi_1^i\varphi_{k+1}\varphi_1^{p-1-i}=0
\end{eqnarray}
for any $k\geq 1$ and $\mathrm{ord}(\varphi_1)=p.$ This leads to the following result:
\begin{theorem}\label{theomcogaking}
Suppose $K,p$ are as above. Let $\varphi\in G_\infty(n,K)$ be a $p$-periodic series. If $\varphi_1$ is a diagonalizable matrix in $GL_n(K),$ then $\varphi$ is conjugate to $\varphi_1,$ i.e. there exists $\phi\in G_\infty(n,K)$ such that $\varphi=\phi^{-1}\varphi_1 \phi.$
\end{theorem}
\begin{proof}
Notice that $(A^{-1}\varphi A)_1=A^{-1}\varphi_1 A$ for any $A\in GL_n(K)$ and conjugacy is an equivalence relation, it is sufficient to prove the case where $\varphi_1$ is a diagonal matrix. 
	
We prove this theorem by constructing a series $\phi\in G_\infty(n,K)$ so that $\varphi=\phi^{-1}\varphi_1 \phi,$ or equivalently, 
$(\phi\varphi-\varphi_1 \phi)_k=0$ for any $k\in\mathbb{N}_{\geq 1}.$ We divide our proof in six steps.
\begin{itemize}
\item[(\romannumeral1).] First, take $\phi_1=id.$ Now assume that $\phi=\sum_{i=1}^k \phi_i,\phi_i\in K[X]_i^n$ satisfies $\varphi_{\leq k}=(\phi^{-1}\varphi_1\phi)_{\leq k}.$ It follows from Lemma \ref{zhankai01} that 
$$[(\phi+\phi_{k+1})\varphi-\varphi_1(\phi+\phi_{k+1})]_{k+1}=(\phi\varphi)_{k+1}+\phi_{k+1}\varphi_1-\varphi_1\phi_{k+1}.$$
Then our goal is to find $\phi_{k+1}\in K[X]_{k+1}^n$ such that 
\begin{eqnarray}\label{goalshi002}
(\phi\varphi)_{k+1}+\phi_{k+1}\varphi_1-\varphi_1\phi_{k+1}=0.
\end{eqnarray}
\item[(\romannumeral2).]  Next, consider the operator 
$$T:K[X]_{k+1}^n\rightarrow K[X]_{k+1}^n,\ \ T(u)=[(\phi+u)^{-1}\varphi_1(\phi+u)]_{k+1}.$$
We claim that 
$$T(u)=\varphi_1 u-u\varphi_1+\varphi_{k+1}-(\phi\varphi)_{k+1}.$$
By Theorem \ref{inverseformula66} and the assumption $\varphi_{\leq k}=(\phi^{-1}\varphi_1\phi)_{\leq k},$ we have 
$$[(\phi+u)^{-1}\varphi_1(\phi+u)]_{\leq k+1}=\varphi_{\leq k}+T(u).$$
 To prove the above claim, it is sufficient to show that 
 \begin{eqnarray}\label{prfclimtu09}
 [\varphi_1(\phi+u)]_{k+1}=[(\phi+u)(\varphi_{\leq k}+\varphi_1 u-u\varphi_1+\varphi_{k+1}-(\phi\varphi)_{k+1})]_{k+1}.
 \end{eqnarray}
The left hand side of (\ref{prfclimtu09}), denoted by $LHS,$ is given by
$$LHS=\varphi_1 u.$$
Using the fact $\phi_1=id$ and Lemma \ref{zhankai01}, the right hand side of (\ref{prfclimtu09}), 
denoted by $RHS,$ is given by 
$$RHS=(\phi_{\leq k}\varphi_{\leq k})_{k+1}+\varphi_1u+\varphi_{k+1}-(\phi\varphi)_{k+1}=\varphi_1u,$$
and hence the equality (\ref{prfclimtu09}) holds. The claim is proved.

\item[(\romannumeral3).]  Let $t(u)=\varphi_1 u-u\varphi_1$ be the linear term of $T(u).$ Since $(\phi+u)^{-1}\varphi_1(\phi+u)$ is always $p$-periodic, we apply (\ref{iffppts56}) to $(\phi+u)^{-1}\varphi_1(\phi+u)$ and $\varphi$ to obtain that
\begin{equation}\label{twoequgscog90}
	\begin{split}
& ((\varphi_{\leq k})^p)_{k+1}+\sum_{i=0}^{p-1}\varphi_1^i\varphi_{k+1}\varphi_1^{p-1-i}=0,\\
& ((\varphi_{\leq k})^p)_{k+1}+\sum_{i=0}^{p-1}\varphi_1^iT(u)\varphi_1^{p-1-i}=0.
\end{split}
\end{equation}
Now define a linear map $R:K[X]_{k+1}^n\rightarrow K[X]_{k+1}^n$ by 
$$R(u)=\sum_{i=0}^{p-1}\varphi_1^i u\varphi_1^{p-1-i}.$$ 
According to (\ref{twoequgscog90}), we see that $R(T(u)-\varphi_{k+1})=0$ for any $u\in K[X]_{k+1}^n.$ 

\item[(\romannumeral4).]  Suppose $\varphi_1=\mathrm{diag}(\omega^{\lambda_1},\cdots,\omega^{\lambda_n}),$ where $\omega$ is a primitive $p$-th root of unity and $\lambda=(\lambda_1,\cdots,\lambda_n)\in\mathbb{Z}^n$ is coprime to $p.$ Note that the set of monomials $\{X_{i,\alpha}\}$ forms a basis for $K[X]_{k+1}^n,$ where the index $i$ runs over $\{1,\cdots,n\}$ and the weight $\alpha$ runs over $\mathbb{N}^n$ with $|\alpha|=\sum \alpha_j=k+1.$ Then 
\begin{equation*}
R(X_{i,\alpha})=\sum_{j=1}^p\omega^{-(\lambda,\alpha)}\cdot \omega^{(\lambda_i-(\lambda,\alpha))j}X_{i,\alpha}
=\left\{\begin{aligned}
	0\ \ \ \ \ & , & \mbox{if}\ p\nmid (\lambda_i-(\lambda,\alpha)),\\
	p\omega^{-(\lambda,\alpha)} X_{i,\alpha}&, & \mbox{otherwise},
\end{aligned}\right. \end{equation*}
where $(\lambda,\alpha)=\sum_{j=1}^n \lambda_j\alpha_j.$ Similarly, it follows that $t(X_{i,\alpha})=(\omega^{\lambda_i}-\omega^{(\lambda,\alpha)})X_{i,\alpha}.$ It is easy to see that $R\circ t=0.$ Since $R$ is linear,  then $R((\phi\varphi)_{k+1})=0.$ 

\item[(\romannumeral5).]  Now let $(\phi\varphi)_{k+1}=\sum a_{i,\alpha}X_{i,\alpha}$ and let $u=\sum u_{i,\alpha}X_{i,\alpha}.$ When $p\nmid (\lambda_i-(\lambda,\alpha)),$ we take $u_{i,\alpha}=\frac{a_{i,\alpha}}{\omega^{\lambda_i}-\omega^{(\lambda,\alpha)}}.$ When $p| (\lambda_i-(\lambda,\alpha)),$ the condition $R((\phi\varphi)_{k+1})=0$ forces that $a_{i,\alpha}=0,$ and we take $u_{i,\alpha}=0.$
Hence $u_0=\sum u_{i,\alpha}X_{i,\alpha}$ satisfies the equation $t(u_0)-(\phi\varphi)_{k+1}=0,$ that is, $u_0$ is a solution to (\ref{goalshi002}).

\item[(\romannumeral6).] Finally, by induction on $k,$ we have constructed $\phi=\sum_{k=1}^\infty \phi_i$ so that $\phi\varphi=\varphi_1 \phi.$ 
\end{itemize}
\end{proof}
\begin{remark}
Theorem \ref{theomcogaking} shows that every $p$-periodic series $\varphi\in G_\infty(n,K)$ with $\varphi_1$ diagonalizable is   
determined by an implicit function $\phi(y)=\varphi_1\phi(x)$ of $x$ and $y$ for some $\phi\in G_\infty(n,K).$ In particular, when $p=2$ and 
$\varphi_1=-id,$ the equation $\phi(x)+\phi(y)=0$ is symmetric about $x$ and $y.$
\end{remark}
\begin{corollary}
If $\varphi\in G_\infty(n,\mathbb{C})$ is periodic, then $\varphi$ is conjugate to $\varphi_1.$
\end{corollary}
\begin{proof}
This is an immediate consequence of Theorem \ref{theomcogaking} and a fact that every periodic matrix in $GL_n(\mathbb{C})$ is diagonalizable.
\end{proof}

We next consider a special case $\varphi_1=\omega\cdot id$ with $\mathrm{ord}(\omega)=p.$ By (\ref{relationpps33}), the equation $((\varphi_{\leq k}+\varphi_{k+1})^p)_{k+1}=0$ is equivalent to 
\begin{eqnarray}\label{impfordbsum}
\left(\sum_{i=0}^{p-1} \omega^{-ki-1-k}\right)\varphi_{k+1}=-((\varphi_{\leq k})^p)_{k+1}.
\end{eqnarray}
Simplify the sum in the left side of (\ref{impfordbsum}) as
\begin{equation*}
\sum_{i=0}^{p-1} \omega^{-ki}=\left\{\begin{aligned}
p \ &,&\ p|k,\\	
0 \ &,&\ \ \mbox{else}.
\end{aligned}\right.\end{equation*}
Thus, $\varphi=\omega\cdot id+\sum_{i=2}^\infty \varphi_i\in G_\infty(n,K)$ is $p$-periodic if and only if 
\begin{eqnarray}\label{iffppcs}
	\varphi_{k+1}=-\frac{\omega}{p}(\varphi_{\leq k}^p)_{k+1},\ \mbox{if }\  p|k;\ \ (\varphi_{\leq k}^p)_{k+1}=0,\ \mbox{else}.
\end{eqnarray}
We assert that the condition $(\varphi_{\leq k}^p)_{k+1}=0,p\nmid k$ in (\ref{iffppcs}) can be removed, that is,
\begin{lemma}\label{lemimpp09}
Suppose $K,p$ and $\omega$ are as above, and $\varphi=\omega\cdot id+\sum_{i=2}^{pk}$ is $p$-periodic in $G_{pk}(n,K)$ for some $k\geq 1.$ Let $\varphi_{pk+1}=-\frac{\omega}{p}(\varphi^p)_{pk+1}.$ Then $((\varphi+\varphi_{pk+1})^p)_{r}=0$ for any $r=pk+2,pk+3,\cdots,pk+p.$
\end{lemma}
\begin{proof}
We proof this lemma by induction on $r.$ Put $\phi=\varphi+\varphi_{pk+1}.$ When $r=pk+2,$ by (\ref{iffppcs}) we see that $(\phi^p)_{pk+1}=0,$ and hence $\phi$ is $p$-periodic in $G_{pk+1}(n,K).$ Let $\psi$ be the inverse of $\phi^{p-1}$ in $G_\infty(n,K).$ It is observed that $\psi_{\leq pk+2}=\phi+\psi_{pk+2}$ since $\phi^{p-1}$ has only one inverse in $G_{pk+1}(n,K).$ It follows from $\phi^{p-1}\psi=\psi\phi^{p-1}=id$ that
\begin{eqnarray}
	0 \!\!\!&=&\!\!\! (\phi^{p-1}\psi)_{pk+2}=(\phi^{p-1}(\phi+\psi_{pk+2}))_{pk+2}\nonumber \\
	\!\!\!&=&\!\!\! \varphi_1^{p-1}\psi_{pk+2}+(\phi^p)_{pk+2}\nonumber \\
	\!\!\!&=&\!\!\! \omega^{p-1}\psi_{pk+2}+(\phi^p)_{pk+2}, \label{imeqa65} \\
	0 \!\!\!&=&\!\!\! (\psi\phi^{p-1})_{pk+2}=((\phi+\psi_{pk+2})\phi^{p-1})_{pk+2}\nonumber \\
	\!\!\!&=&\!\!\! (\phi^p)_{pk+2}+\psi_{pk+2}\varphi_1^{p-1} \nonumber \\
	\!\!\!&=&\!\!\! \omega^{-2}\psi_{pk+2}+(\phi^p)_{pk+2}.\label{imeqa66}
\end{eqnarray}
Relying on (\ref{imeqa65}) and (\ref{imeqa66}), we find that $\psi_{pk+2}=(\phi^p)_{pk+2}=0.$ 

Now assume that $(\phi^p)_r=0$ for $r=pk+2,\cdots,pk+s-1,$ where $2<s\leq p.$ Then $\phi$ is $p$-periodic in $G_{pk+s-1}(n,K),$ and hence $\psi_{\leq pk+s}=\phi+\psi_{pk+s}.$ With the same argument
as above, we obtain
\begin{eqnarray}
0 \!\!\!&=&\!\!\! (\phi^{p-1}\psi)_{pk+s}=\omega^{p-1}\psi_{pk+s}+(\phi^p)_{pk+s}, \label{imeqa67} \\
0 \!\!\!&=&\!\!\! (\psi\phi^{p-1})_{pk+s}=\omega^{-s}\psi_{pk+s}+(\phi^p)_{pk+s}.\label{imeqa68}.
\end{eqnarray}
Then $\psi_{pk+s}=(\phi^p)_{pk+s}=0$ as $\omega^{s-1}\ne 1.$ 
\end{proof}
Thus, we have
\begin{theorem}\label{thm666pp}
	Suppose $K,p,\omega$ are as above.  Let $\varphi=\omega\cdot id+\sum_{i=2}^\infty\varphi_i\in G_\infty(n,K).$ Then $\varphi$ is $p$-periodic if and only if $\varphi_{pk+1}=-\frac{\omega}{p}((\varphi_{\leq pk})^p)_{pk+1}$ for any $k\geq 1.$
\end{theorem}
\begin{proof}
It is an immediate consequence of (\ref{iffppcs}) and Lemma \ref{lemimpp09}. 
\end{proof}
\subsection{Examples}
\begin{example}\label{exmpp2}
	Consider the simplest case $n=1$ and $p=2.$ The recurrence relation for a $2$-periodic series $\varphi=-x+\sum_{k=2}^\infty a_kx^k\in G_\infty(1,\mathbb{Q})$ becomes $\varphi_{2k+1}=\frac{1}{2}((\varphi_{\leq 2k})^2)_{2k+1}.$ The first few relations are
	\begin{equation*}
		\begin{split}
			 a_3=& -a_2^2,\\
			 a_5=&2a_2^4-3a_2a_4,\\
			a_7=&-13a_2^6+18a_2^3a_4-2a_4^2-4a_2a_6,\\
			a_9=& 145 a_2^8-221 a_2^5a_4+50a_2^2a_4^2+35a_2^3a_6-5 a_2a_8-5 a_4 a_6,\\
		  a_{11}=&-2328 a_2^{10}+3879 a_2^7a_4-561 a_2^5a_6-1263 a_2^4a_4^2+61 a_2^3 a_8 \\
	     &	+171 a_2^2a_4 a_6  +55a_2a_4^3-6a_2a_{10}-3 a_6^2-6 a_4 a_8	,\\
			\cdots &
		\end{split}
	\end{equation*}
\end{example}
\begin{example}
Examine the case $n=1$ and $p=3.$ Let $\varphi=\omega x+\sum_{k=2}^\infty a_kx^k$ be a $3$-periodic series in $G_\infty(1,\mathbb{Q}[\omega]),$ where $\omega=e^{\frac{2\pi i}{3}}.$ The first few relations are given by
		\begin{equation*}
			\begin{split}
			a_4=&\frac{1-4\omega}{3}a_2^3-\frac{7+8\omega}{3}a_2a_3,\\
			 a_7=&(3\omega-1)a_3^3+(4-8\omega)a_2^2a_5+\frac{116-68\omega}{9}a_2^6-\frac{10+14\omega}{3}a_2a_6 \\
			 &  -\frac{13+11\omega}{3}a_3a_5-\frac{154+332\omega}{9}a_2^4a_3-\frac{79+74\omega}{9}a_2^3a_3^2,\\
			 a_{10}=& \frac{22586+5656\omega}{27}a_2^9+\frac{1461-14576\omega}{9} a_2^7a_3
		         -\frac{14446+14501\omega}{27}a_2^5a_3^2   \\  
		          &  -\frac{8891+12490\omega}{27}a_2^6a_3^2  +
		        \frac{8-1604\omega}{27} a_2^3a_3^3 -\frac{8101+8\omega}{27} a_2^4a_3^3  \\
		       &  +\frac{374+505\omega}{3}a_2a_3^4  + \frac{1247-182\omega}{3}a_2^5a_5 -\frac{2131+4319\omega}{9}a_2^3a_3a_5  \\
		       & +\frac{-338+80\omega}{9} a_2a_3^2a_5+ (7-15\omega) a_2a_5^2 -\frac{103+1154\omega}{9} a_2^4a_6  \\
		        & -\frac{545+361\omega}{9} a_2^2a_3a_6 -(7+22\omega) a_3^2a_6 -\frac{16+17\omega}{3}a_5 a_6 \\
		        & +(10-12\omega) a_2^2 a_8-\frac{19+14\omega}{3}a_3 a_8 -\frac{13+20\omega}{3}a_2 a_9 ,\\
		        \cdots &
		       \end{split}
	\end{equation*}
\end{example}
\begin{example}
Look at the case $n=1$ and $p=4.$ Let $\varphi=i x+\sum_{k=2}^\infty a_kx^k$ be a $4$-periodic series in $G_\infty(1,\mathbb{Q}[i]),$ where $i=\sqrt{-1}.$ The first two relations are given by
\end{example}
\begin{equation*}
\begin{split}
a_5=& \frac{-3+5i}{2}a_2^4+\frac{12+5i}{2}a_2^2 a_3-\frac{3 i}{2}a_3^2+(1-3 i) a_2 a_4, \\
a_9=& \frac{743+118i}{4}a_2^8-(35+578 i) a_2^6 a_3+\frac{-3504+1307 i}{8}a_2^4 a_3^2+\frac{172+215 i}{4}a_2^2 a_3^3 \\
 & -\frac{65i}{8}a_3^4+\frac{-345+23 i}{2} a_2^5 a_4+\frac{119+332i}{2} a_2^3 a_3 a_4+\frac{2+99 i}{2} a_2 a_3^2 a_4 \\
 & +\frac{33-25 i}{2} a_2^2 a_4^2+\frac{32-7 i}{2} a_3 a_4^2 -(35-25 i) a_2^3 a_6+(32+11 i) a_2 a_3 a_6 \\
 & -(1+5 i) a_4 a_6+\frac{24+27 i}{2} a_2^2 a_7-5 i a_3 a_7+(3-5 i) a_2 a_8, \\
 \cdots &
\end{split}
\end{equation*}
\section{An open problem}
At the end of this paper, we propose a conjecture about periodic series over $\mathbb{C}.$ Consider a special case in Example \ref{exmpp2} . Let 
$$f_a(x)=-x+ax^2+\sum_{k=1}^\infty c_ka^{2k}x^{2k+1}$$
be a $2$-periodic series in $G_\infty(1,\mathbb{C}),$ where $\{c_k\}_{k=1}^\infty$ is a sequence of integers with first few items 
$$c_1=-1,\ c_2=2,\ c_3=-13,\ c_4=145,\ c_5=-2328,\ c_6=49784,\ \cdots$$
See the sequence \cite{A107699} for more integers. Take $a=1.$ The series $\sum_{k=1}^\infty c_kx^k$ appears to be divergent for any $x\in\mathbb{C}^*.$ Moreover, we conjecture that
\begin{conjecture}
Let $p\geq 2$ and $\omega=\exp({\frac{2\pi i}{p}})\in\mathbb{C}.$ Suppose 
	$$\varphi=\omega\cdot id+\sum_{j=2}^p\varphi_j+\sum_{k=1}^\infty\varphi_{pk+1}\in G_\infty(n,\mathbb{C})$$
	 is $p$-periodic. If $\varphi\ne\varphi_1,$ then the radius of convergence of $\varphi$ is zero.
\end{conjecture}

\emph{Acknowledgements.} The author would like to thank Yue Chen for many useful conversations and comments during this work.

%\bibliographystyle{plain}
%\bibliography{refe}

\end{document}